\documentclass[oneside]{amsart}

\usepackage{graphicx}
\usepackage[left,modulo]{lineno}
\linenumbers

\usepackage[T1]{fontenc}
\usepackage[utf8]{inputenc}
\usepackage[english,french]{babel}
\usepackage{amsmath,amssymb,amsthm,graphicx,pstricks}
\usepackage{textcomp}
\usepackage{verbatim}

\newtheorem{theorem}{Theorem}[section]
\newtheorem{lemma}[theorem]{Lemma}
\newtheorem{proposition}[theorem]{Proposition}
\newtheorem{corollary}[theorem]{Corollary}

\theoremstyle{definition}
\newtheorem{definition}[theorem]{Definition}

\theoremstyle{remark}
\newtheorem{remark}[theorem]{Remark}

\numberwithin{equation}{section}
\numberwithin{figure}{section} 

\newcommand{\alias}[2]{
 \providecommand{#1}{}
 \renewcommand{#1}{#2}
}

\newcommand{\ind}[1]{\mathbf{1}_{#1}}

\newcommand{\R}{\mathbb{R}}

\newcommand{\N}{\mathbb{N}}
\newcommand{\Z}{\mathbb{Z}}

\alias{\HNT}{H-non-tangentiel}
\alias{\cone}{sansnom}

\alias{\majN}{\mathcal{N}}
\alias{\majL}{\mathcal{L}}
\alias{\majJ}{\mathcal{J}}
\alias{\majA}{\mathcal{A}}
\alias{\majB}{\mathcal{B}}
\alias{\majC}{\mathcal{C}}
\alias{\majM}{\mathcal{M}}
\alias{\majS}{\mathcal{S}}
\alias{\majF}{\mathcal{F}}
\alias{\majG}{\mathcal{G}}
\alias{\majH}{\mathcal{H}}
\alias{\majD}{\mathcal{D}}
\alias{\majR}{\mathcal{R}}
\alias{\majV}{\mathcal{V}}
\alias{\majW}{\mathcal{W}}

\alias{\mA}{\mathfrak{A}}
\alias{\mB}{\mathfrak{B}}
\alias{\mL}{\mathfrak{L}}
\alias{\mK}{\mathfrak{K}}
\alias{\mD}{\mathfrak{D}}
\alias{\mR}{\mathfrak{R}}
\alias{\mM}{\mathfrak{M}}
\alias{\mN}{\mathfrak{N}}
\alias{\mG}{\mathfrak{G}}
\alias{\mV}{\mathfrak{V}}
\alias{\mW}{\mathfrak{W}}
\alias{\mS}{\mathfrak{S}}
\alias{\mH}{\mathfrak{H}}
\alias{\mh}{\mathfrak{h}}
\alias{\mg}{\mathfrak{g}}

\alias{\Z}{\mathbb{Z}}
\alias{\Q}{\mathbb{Q}}
\alias{\R}{\mathbb{R}}
\alias{\T}{\mathbb{T}}
\alias{\E}{\mathbb{E}}
\alias{\P}{\mathbb{P}}
\alias{\H}{\mathbb{H}}
\alias{\G}{\mathbb{G}}
\alias{\M}{\mathbb{M}}
\alias{\V}{\mathbb{V}}
\alias{\W}{\mathbb{W}}
\alias{\T}{\mathbb{T}}
\alias{\A}{\mathbb{A}}
\alias{\B}{\mathbb{B}}

\alias{\diam}{\mathrm{diam}}

\begin{document}

\title{Boundary behaviour of harmonic functions on hyperbolic manifolds}

\author{Camille PETIT}
\address{University of Jyväskylä\\
 Department of Mathematics and Statistics\\
P.O. Box 35 (MaD)\\
FI-40014 University of Jyväskylä\\
Finland }
\email{camille.c.petit@jyu.fi}




\selectlanguage{english}
\begin{abstract}
Let $M$ be a complete simply connected manifold which is in addition Gromov hyperbolic, coercive and roughly starlike. For a given harmonic function on $M$, a local Fatou Theorem and a pointwise criteria of non-tangential convergence coming from the density of energy are shown: at almost all points of the boundary, the harmonic function converges non-tangentially if and only if the supremum of the density of energy is finite.  As an application of these results, a Calder\'on-Stein Theorem is proved, that is, the non-tangential properties of convergence, boundedness and finiteness of energy are equivalent at almost every point of the boundary.

\end{abstract}

\maketitle

\tableofcontents


\section{Introduction}


The interplay between the geometry of a complete Riemannian manifold $M$ and the existence of non-constant harmonic functions on $M$ has been studied by researcher in geometric analysis for decades. On one hand, S.T.~Yau \cite{Yau75} proved that, on a complete Riemannian manifold $M$ of non-negative Ricci curvature, every positive harmonic function is constant. On the other hand, in the middle of the eighties, M.T.~Anderson and R.~Schoen \cite{AS85} provided a complete description of the space of non-negative harmonic functions for manifolds of pinched negative curvature. They proved the identification between the sphere at infinity and the Martin boundary, that is, the boundary allowing an integral representation of non-negative harmonic functions by measures on the boundary. Major contributions to this latter issue were given by A.~Ancona in a series of papers \cite{Anc87, Anc88, Anc90} and will be discussed later in this article.

The study of non-tangential convergence of harmonic functions goes back to 1906 with P.~Fatou's seminal paper \cite{Fat06}, where the following result is proved:

{\it  Any positive harmonic function on the unit disc admits  non-tangential limits at almost every point $\theta$ of the boundary circle.} 

Recall that a function is said to \textit{converge non-tangentially at $\theta$} if it has a finite limit at $\theta$ on every non-tangential cone with vertex $\theta$.
Fatou type theorems have been proved in many different contexts since then and in particular on Gromov hyperbolic graphs and manifolds \cite{Anc88, Anc90}.  One of the issues is to replace the global positivity condition by local criteria of non-tangential convergence. Of special interest are two criteria which have been intensively studied: the criterion of non-tangential boundedness \cite{Pri16, Cal50b} and the criterion of finiteness of non-tangential integral area \cite{MZ38, Spe43, Cal50a, Ste61}, also called the criteria of Lusin area. A function is \textit{non-tangentially bounded at a boundary point $\theta$} if it is bounded on every non-tangential cone with vertex $\theta$. Similarly, a function is of \textit{finite non-tangential integral area at $\theta$} if on every non-tangential cone with vertex $\theta$, it has a finite area integral. Calder\' on-Stein's Theorem (\cite{Cal50b,Ste61}) asserts that for a harmonic function in the Euclidean half-space, notions of non-tangential convergence, non-tangential boundedness and finiteness of non-tangential area integral coincide at almost all points of the boundary. These results were reproved by J.~Brossard (\cite{Bro78}) using Brownian motion. J.~Brossard also stated in \cite{Bro88} the criterion of the density of the area integral, a notion first introduced by R.F.~Gundy \cite{Gun83}.

 As noticed by A. Kor\' anyi, hyperbolic spaces provide a convenient framework for studying Calder\'on-Stein like results. A first reason is that there should exist a lot of non-constant harmonic functions under negative curvature assumptions, whereas they are in some sense rare on manifolds satisfying non-negative curvature assumptions. Another reason is that several notions have simpler and more natural expressions in the setting of hyperbolic spaces. When we equip the Euclidean half-space with the hyperbolic Poincar\'e metric, Euclidean notions of non-tangential cone with vertex $\theta$:
 $$\Gamma_\alpha^\theta := \{ (x,y)\in \R^\nu \times \R_+ \, | \, |x-\theta|<ay<a \},$$
 of non-tangential area integral and of density of  area integral:
 $$\int_{\Gamma_\alpha^\theta} | \nabla u(x,y)|^2 y^{1-\nu} dx dy \ \text{ and } \ \frac 12 \int_{\Gamma_\alpha^\theta} y^{1-\nu} \Delta |u-r|(dxdy),$$
 turn out to be respectively tubular neighborhoods of geodesic rays starting at a base point $o$:
 $$\Gamma_c^\theta:= \{ z \, | \, \exists \, \gamma \text{ a geodesic ray from } o \text{ to } \theta \text{ such that } d(z,\gamma)<c \},$$
 a true energy and a density of energy:
 $$J_c^\theta:=\int_{\Gamma_c^\theta} | \nabla u|^2 d\nu \ \text{ and } \ D_c^r(\theta):=-\frac 12\int_{\Gamma_c^\theta} \Delta |u-r |(dx) .$$

Following this philosophy, Calder\'on-Stein's result was extended to the framework of Riemannian manifolds of pinched negative curvature (\cite{Mou95}) and trees (\cite{Mou00, AP08, Mou10, Pic10}). The criterion introduced by J.~Brossard of the density of area integral was also addressed in \cite{Mou07} for Riemannian manifolds of pinched negative curvature. The present paper was motivated by the question whether these kinds of results also hold for Gromov hyperbolic spaces. In a recent paper \cite{Pet12}, we proved the criteria of non-tangential boundedness in the framework of Gromov hyperbolic graphs. The aim of this paper is to deal with the different criteria presented above in the case of Gromov hyperbolic manifolds.

Let us now describe the main results of this paper. We first introduce briefly the geometric setting. The geometric notions will be defined precisely in section \ref{sect:preliminaries}. We say that a Riemannian manifold $M$ is \textit{roughly-starlike} if there exist a constant $K\geq 0$ and a base point $o\in M$ such that every point $x\in M$ is within a distance at most $K$ from a geodesic ray starting at $o$. A Riemannian manifold $M$ of dimension $n$ has \textit{bounded local geometry} provided about each $x\in M$, there is a geodesic ball $B(x,r)$ (with $r$ independent of $x$) and a diffeomorphism $F:B(x,r)\to \R^n$ with
$$\frac 1c \cdot d(y,z)\leq \| F(y)-F(z)\| \leq c \cdot d(y,z)$$
for all $y,z\in B(x,r)$, where $c$ is independent of $x$. It is worth mentioning that $M$ has bounded local geometry if $M$ has Ricci curvature and injectivity radius bounded from below.
Following the terminology of A.~Ancona, $M$ is called \textit{coercive} if it has bounded local geometry and if the bottom $\lambda_1(M)$ of the spectrum is positive. We say that a complete, simply connected Riemannian manifold $M$ satisfies condition ($\clubsuit$) if in addition $M$ is coercive, roughly starlike and Gromov hyperbolic.

We first prove a local Fatou Theorem for Riemannian manifolds satisfying conditions ($\clubsuit$). Let $U$ be an open set in $M$ and let us denote by $\partial M$ the geometric boundary of $M$. We say that a point $\theta\in \partial M$ is \textit{tangential for $U$} if for all $c>0$, the set $\Gamma_c^\theta \setminus U$ is bounded.

\begin{theorem}  \label{thm:local Fatou}
Let $M$ be a manifold satisfying conditions ($\clubsuit$) and let $U$ be an open subset of $M$. If $u$ is a non-negative harmonic function on $U$, then for $\mu$-almost all $\theta$ that is tangential for $U$, the function $u$ converges non-tangentially at $\theta$.
\end{theorem}

The measure $\mu$ on $\partial M$ is the \textit{harmonic measure}.
The proof follows the approach by F.~Mouton \cite{Mou07}. In Mouton's proof, geometry comes in at some key points by use of comparison Theorems in pinched negative curvature. 
As a Corollary of Theorem \ref{thm:local Fatou}, we deduce that a harmonic function converges non-tangentially at almost all points where it is non-tangentially bounded from below (or above).

\begin{corollary}  \label{cor:bnd from below}
Let $M$ be a manifold satisfying conditions ($\clubsuit$) and let $u$ be a harmonic function on $M$. Then, for $\mu$-almost all $\theta\in \partial M$, the following properties are equivalent:
\begin{enumerate}
 \item The function $u$ converges non-tangentially at $\theta$.
 \item The function $u$ is non-tangentially bounded from below at $\theta$.
 \item There exists $c>0$ such that $u$ is bounded from below on $\Gamma_c^\theta$.
\end{enumerate} 
\end{corollary}

Then we focus on the density of energy. In \cite{Bro88}, J.~Brossard proved that for a harmonic function $u$ on the Euclidean half-space, at almost all points of the boundary, $u$ converges non-tangentially if and only if the supremum over $r\in \R$ of the density of area integral on the level set $\{ u=r\}$ is finite. In \cite{Mou07}, F.~Mouton proved that for a harmonic function $u$ on a Riemannian manifold of pinched negative curvature, $u$ converges non-tangentially at almost all points of the boundary where the density of energy on the level set $\{ u=0\}$ is finite, providing a partial geometric analogue of Brossard's Theorem. We focus here on Gromov hyperbolic manifolds and prove an analogue for the density of energy of Brossard's result. It generalizes and strengthens Mouton's Theorem.

\begin{theorem}  \label{thm:density energy}
Let $M$ be a manifold satisfying conditions ($\clubsuit$), $c>0$, and let $u$ be a harmonic function on $M$. Then, for $\mu$-almost all $\theta\in \partial M$, the following properties are equivalent:
\begin{enumerate}
 \item The function $u$ converges non-tangentially at $\theta$.
 \item $\sup_{r\in \R} D_c^r(\theta)<+\infty$.
 \item $D_c^0(\theta)<+\infty$.
\end{enumerate} 
\end{theorem}

As a Corollary, we prove the criterion of the finiteness of the non-tangential energy. 

\begin{corollary}  \label{thm:energy}
Let $M$ be a manifold satisfying conditions ($\clubsuit$) and let $u$ be a harmonic function on $M$. Then, for $\mu$-almost all $\theta\in \partial M$, the following properties are equivalent:
\begin{enumerate}
 \item The function $u$ converges non-tangentially at $\theta$.
 \item The function $u$ has finite non-tangential energy at $\theta$.
\end{enumerate} 
\end{corollary}

Corollaries \ref{cor:bnd from below} and \ref{thm:energy} together yield in particular the Calder\' on-Stein Theorem on Gromov hyperbolic manifolds.

This paper is organized as follows. In section \ref{sect:preliminaries}, we present the geometric framework of the results, recalling briefly some properties of Gromov hyperbolic metric spaces, the definition of a roughly starlike manifold and of a coercive manifold. In section \ref{sect:BM and conditioning} we discuss some basics related to Brownian motions needed later on, in particular the martingale property and the Doob's h-process method, allowing to condition Brownian motion to exit the manifold at a fixed point of the boundary. In section \ref{sect:Harnack inequalities}, we recall the different Harnack inequalities needed later on. Section \ref{sect:NT behaviour of Brownian motion}  is devoted to the proofs of several lemmas ensuing Harnack inequalities and crucial in the proofs of the main results. Finally, in section \ref{sect:local Fatou theorem} we prove Theorem \ref{thm:local Fatou} and Corollary \ref{cor:bnd from below} and in section \ref{sect:density of energy}, we prove Theorem \ref{thm:density energy} and Corollary \ref{thm:energy}.


\section{Preliminaries}   \label{sect:preliminaries}

From now on, $M$ denotes a complete simply connected Riemannian manifold of dimension $n\geq 2$ and $d$ denotes the usual Riemannian distance on $M$. Let $\Delta$ denote the Laplace-Beltrami operator on $M$, by $G$ the associated Green function. A function $u:M\to \R$ is called \textit{harmonic} if $\Delta u=0$. The Green function $G$ is finite outside the diagonal, positive, symmetric and for every $y\in M$, the function $x\mapsto G(x,y)$ is harmonic on $M\setminus \{ y\}$. We will make additional geometric assumptions on $M$, which will be described in the following paragraphs. 

\subsection{Gromov hyperbolic spaces}

Gromov hyperbolic spaces have been introduced by M.~Gromov in the 80's (see for instance \cite{Gro81, Gro87}). These spaces are naturally equipped with a geometric boundary. There exists a wide literature on Gromov hyperbolic spaces, see \cite{GdlH90, BH99} for nice introductions. We introduce here only the properties of these spaces which will be used in the following.

Let $(X,d)$ denote a metric space. The \textit{Gromov product} of two points $x,y\in X$ with respect to a basepoint $o\in X$ is defined by
$$(x,y)_o:= \frac 12 \left[ d(o,x)+d(o,y)-d(x,y) \right].$$
Notice that $0\leq (x,y)_o\leq \min \{ d(o,x), d(o,y)\}$ and that if $o'\in X$ is another basepoint, then for every $x,y\in X$,
$$|(x,y)_o-(x,y)_{o'} |\leq d(o,o').$$

\begin{definition}
A metric space $(X,d)$ is called \textit{Gromov hyperbolic} if there exists $\delta\geq 0$ such that for every $x,y,z\in X$ and every basepoint $o\in X$,
\begin{equation}  \label{ineq:hyperbolicity 1}
(x,z)_o\geq \min \{ (x,y)_o, (y,z)_o\} -\delta.
\end{equation}
For a real $\delta\geq 0$, we say that $(X,d)$ is \textit{$\delta$-hyperbolic} if inequality (\ref{ineq:hyperbolicity 1}) holds for all $x,y,z,o\in X$.
\end{definition}

\begin{remark}
From now on, when considering a Gromov hyperbolic metric space, we will always assume, without loss of generality, that inequality (\ref{ineq:hyperbolicity 1}) holds with $\delta$ an integer greater than or equal to 3.
\end{remark}

The definition of Gromov hyperbolicity makes sense in every metric space. When the metric space is Gromov hyperbolic and geodesic, the Gromov product $(x,y)_o$ may be seen as a rough measure of the distance between $o$ and any geodesic segment between $x$ and $y$. More precisely, if $\gamma$ is a geodesic segment between $x$ and $y$, we have
$$d(o,\gamma)-2\delta \leq (x,y)_o \leq d(o,\gamma).$$

We now describe the geometric boundary of a Gromov hyperbolic space. Let $(X,d)$ be a $\delta$-hyperbolic metric space and fix a basepoint $o\in X$. A sequence $(x_i)_i$ in $X$ \textit{converges at infinity} if
$$\lim_{i,j\to +\infty} (x_i,x_j)_o=+\infty.$$
This condition is independent of the choice of the basepoint. Two sequences $(x_i)_i$ and $(y_j)_j$ converging at infinity are called \textit{equivalent} if $\lim_{i\to +\infty} (x_i,y_i)_o=+\infty$. This defines an equivalence relation on sequences converging at infinity. The \textit{geometric boundary} $\partial X$ is the set of equivalence classes of sequences converging at infinity. In order to fix an appropriate topology on $\overline X:= X\cup \partial X$, we extend the Gromov product to the boundary. Let us say that for a point $x\in X$, a sequence $(x_i)\in X^\N$ is in the class of $x$ if $x_i\to x$. We then define
$$(x,y)_o:= \sup \liminf_{i,j\to\infty} (x_i,y_j)_o,$$
where the supremum is taken over all sequences $(x_i)$ in the class of $x\in \overline X$ and $(y_j)$ in the class of $y\in \overline X$.
The inequality
\begin{equation}  \label{ineq:hyperbolicity 2}
(x,z)_o\geq \min \{ (x,y)_o, (y,z)_o \} -2\delta
\end{equation}
holds for every $x,y,z \in \overline X$. If in addition $(X,d)$ is geodesic, then for every $x \in X$, $\xi\in \partial X$ and every geodesic ray $\gamma$ from $o$ to $\xi$, we have
\begin{equation}    \label{ineq:Gromov product and geodesics}
d(x,\gamma)-2\delta \leq (o,\xi)_x \leq d(x,\gamma)+2\delta.
\end{equation}
 For a real $r\geq 0$ and a point $\xi\in\partial X$, denote $V_r(\xi):=\{ y\in \overline X \, | \, (\xi,y)_o\geq r\}$. We then equip $\overline X$ with the unique topology containing open sets of $X$ and admitting the sets $V_r(\xi)$ with $r\in \Q_+$ as a neighborhood base at any $\xi\in \partial X$. This provides a compactification $\overline X$ of $X$.

\subsection{Roughly starlike manifolds}

We will assume the manifold $M$ to be roughly starlike. From now on, fix a basepoint $o\in M$.

\begin{definition}
A complete Riemannian manifold $M$ is called \textit{roughly starlike with respect to the basepoint $o\in M$} if there exists $K\geq 0$ such that for every point $x\in M$, there exists a geodesic ray $\gamma$ starting at $o$ and within a distance at most $K$ from $x$.
\end{definition}

We will abbreviate to roughly starlike if there is no risk of ambiguity.
Let us notice that if $M$ is $\delta$-hyperbolic and $K$-roughly starlike with respect to $o$, then $M$ is $K'$-roughly starlike with respect to $o'$, with $K'=K'(d(o,o'),\delta,K)$. The "roughly starlike" assumption has previously been used by A.~Ancona \cite{Anc88} and by M.~Bonk, J.~Heinonen and P.~Koskela \cite{BHK01}.

Recall that a complete manifold $M$ is said to have a \textit{quasi-pole} in a compact set $\Omega\subset M$ if there exists $C>0$ such that each point of $M$ lies in a $C$-neighborhood of some geodesic ray emanating from $\Omega$. If $M$ is roughly starlike with respect to $o$, then $M$ has a quasi-pole at $o$.

\subsection{Coercive manifolds}

As explained in the introduction, a manifold $M$ of dimension $n$ has \textit{bounded local geometry} provided about each $x\in M$, there is a geodesic ball $B(x,r)$ (with $r$ independent of $x$) and a diffeomorphism $F:B(x,r)\to \R^n$ with
$$\frac 1c d(y,z)\leq \| F(y)-F(z)\| \leq cd(y,z)$$
for all $y,z\in B(x,r)$, where $c$ is independent of $x$.

The manifold $M$ is \textit{coercive} if it has bounded local geometry and if the bottom $\lambda_1(M)$ of the spectrum is positive. Recall that the bottom of the spectrum of the Laplacian $\Delta$ is defined by
$$\lambda_1(M):=\inf_\phi \frac{\int_M \| \nabla \phi \|^2}{\int_M \phi^2},$$
where $\phi$ ranges over all smooth functions with compact support on $M$. Notice that for a manifold of bounded local geometry, $\lambda_1(M)>0$ if and only if its Cheeger constant is positive (see \cite{Bus82}).

It is worth mentioning that in general, Gromov hyperbolicity does not imply positivity of the bottom of the spectrum. In \cite{Cao00}, J.~Cao gave conditions for a Gromov hyperbolic, roughly starlike manifold with bounded local geometry to have positive bottom of the spectrum.

\subsection{Comments}   \label{subsect:comments}

Recall that we say that a complete, simply connected Riemannian manifold $M$ satisfies condition ($\clubsuit$) if in addition $M$ is coercive, roughly starlike and Gromov hyperbolic.
On one hand, when the manifold $M$ is Gromov hyperbolic, we can consider its geometric boundary as defined above. On the other hand, we can also consider its Martin boundary, which is natural when dealing with non-negative harmonic functions. In \cite{Anc90}, A.~Ancona proved that for a manifold satisfying conditions ($\clubsuit$) (and even without roughly starlike assumption), the geometric compactification and the Martin compactification are homeomorphic.

The following Proposition is a consequence of conditions ($\clubsuit$). It yields a uniformity in the behaviour of the Green function $G$ and will be useful in the following.

\begin{proposition}[\cite{Anc90}, page 92]   \label{prop:exponential decay}
If $M$ satisfies conditions ($\clubsuit$), there exist two positive constants $C_1=C_1(M)$ and $c_1=c_1(M)$ such that for every $x,y\in M$ with $d(x,y)\geq 1$, we have
\begin{equation*}
G(x,y)\leq C_1 \exp(-c_1d(x,y)).
\end{equation*}
\end{proposition}


\section{Brownian motion and conditioning}   \label{sect:BM and conditioning}

Following the philosophy of J.~Brossard \cite{Bro88}, our methods use Brownian motion and the connection between harmonic functions and Brownian motion given by the martingale property. We describe here the "Brownian material" needed in the proofs.

\subsection{Brownian motion}

The \textit{Brownian motion} $(X_t)$ on $M$ is defined as the diffusion process associated with the Laplace-Beltrami operator $\Delta$. If $M$ satisfies condition~($\clubsuit$), Brownian motion is defined for every $t\in \R_+$ (\cite{Anc90}, page 60). Choosing $\Omega:=\mathcal C(\R_+,M)$ as the probability space, for every $t\in \R_+$, $X_t$ is a random variable on $\Omega$, with values in $M$, and for every $\omega \in \Omega$, $t\mapsto X_t(\omega)$ is a continuous function, that is a path in $M$. If we consider Brownian motion starting at a fixed point $x\in M$, we obtain a probability $\mathbb P_x$ on $\Omega$. 

An important property of the Martin boundary (which in our case coincides with the geometric boundary, see section \ref{subsect:comments}) is that for $\mathbb P_x$-almost every trajectory $\omega \in \Omega$, there exists a boundary point $\theta\in \partial M$ such that $\lim_{t\to +\infty} X_t(\omega)=\theta$. Let us denote by $X_\infty(\omega)$ the $\partial M$-valued random variable such that Brownian motion converges $\mathbb P_x$-almost surely to $X_\infty$ for all $x\in M$. The \textit{harmonic measure at $x$}, denoted by $\mu_x$, is the distribution of $X_\infty$ when Brownian motion starts at $x$. All the measures $\mu_x, x\in M$ on $\partial M$ are equivalent. This gives rise to a notion of $\mu$-negligibility. Defining the \textit{Poisson kernel} $K(x,\theta)$ as limit of the Green kernels $\lim_{y\to \theta} \frac{G(x,y)}{G(o,y)}$, the Radon-Nikodym derivative of harmonic measure is given by
$$K(x,\theta)=(d\mu_x /d\mu_o)(\theta).$$

The \textit{martingale property} (see \cite{Dur84}) is a crucial tool in our methods: for a function $f$ of class $C^2$, 
$$f(X_t)+\frac 12 \int_0^t \Delta f(X_s)ds$$
is a local martingale with respect to probabilities $(\mathbb P_x)_x$. Hence if $u$ is harmonic, $(u(X_t))$ is a local martingale.

\subsection{Conditioning}

As claimed above, Brownian motion converges almost surely to a boundary point. \textit{Doob's h-process method} \cite{Doo57} allows to condition Brownian motion to "exit" the manifold at a fixed point $\theta\in\partial M$. For every $x\in M$, we obtain a new probability $\mathbb P_x^\theta$ on $\Omega$, whose support is contained in the set of trajectories starting at $x$ and converging to $\theta$ (see \cite{Bro78, Mou94}). This probability satisfies a \textit{strong Markov property} and an \textit{asymptotic zero-one law}. For all $N\in \N$, denote by $\tau_N$ the exit time of the ball $B(o,N)$ and by $\mathcal{F_{\tau_N}}$ the associated $\sigma$-algebra. Let $\mathcal F_\infty$ be the $\sigma$-algebra generated by $\mathcal F_{\tau_N}, N\in \N$.
We can reconstruct the probability $\mathbb P_x$ with the conditioned probabilities: for a $\mathcal F_\infty$-measurable random variable $F$,

\begin{equation}   \label{eq:conditioning}
\mathbb E_x[F]=\int_{\partial M} \mathbb E_x^\theta[F] d\mu_x(\theta).
\end{equation}

\subsection{Stochastic convergence}

The behaviour of a harmonic function along trajectories of Brownian motion is easily studied by means of martingale theorems.
For a function $f$ on $M$, let us define the following event:
$$\majL_f^{**}:=\{\omega\in\Omega \, | \, \lim_{t\to\infty} f(X_t(\omega)) \text{ exists and is finite} \};$$
The asymptotic zero-one law implies that the quantity $\mathbb{P}_x^\theta(\majL_f^{**})$ does not depend on $x$ and has value $0$ or $1$. In the second case, we say that $f$ \textit{converges stochastically at $\theta$}. In the same way we say that $f$ is \textit{stochastically bounded at $\theta$} if $\mathbb P_o^\theta$-a.s., $f(X_t)$ is bounded, and that $f$ is of \textit{finite stochastic energy at $\theta$} if $\mathbb P_o^\theta$-a.s., $\int_0^{+\infty} | \nabla f(X_t(\omega))|^2 dt<+\infty$. By the martingale property and martingale theorems, F.~Mouton \cite{Mou95} proved that for a harmonic function, the set of points $\theta\in\partial M$ where there is respectively stochastic convergence, stochastic boundedness and finiteness of stochastic energy, are $\mu$-almost equivalent, that is they differ by a set of $\mu$-measure zero.

When the harmonic function $u$ is bounded, non-tangential and stochastic convergences at $\mu$-almost all points of the boundary are automatic (\cite{Anc90}):

\begin{lemma}	\label{lem:bounded harmonic functions}
A bounded harmonic function $u$ on $M$ converges non-tangentially and stochastically at $\mu$-almost all points $\theta\in\partial M$ and the unique function $f\in L^\infty(\partial M,\mu)$ such that 
$$u(x)=\int_{\partial M} f(\theta )d\mu_x (\theta )=\mathbb{E}_x[ f(X_{\infty})]$$
is $\mu$-a.e. the non-tangential and stochastic limit of $u$.
\end{lemma}


\section{Harnack inequalities}    \label{sect:Harnack inequalities}

We will use comparison theorems between non-negative harmonic functions several times. The first one is the usual \textit{Harnack inequality on balls}, sometimes called uniform Harnack (see \cite{Anc90} page 21 and \cite{CY75}):

\begin{theorem}[Harnack on balls]     \label{thm:Harnack on balls}
Let $r>0$ and $R>0$ such that $r<R$. There exists a constant $C>0$ such that for all points $x\in M$ and all non-negative harmonic functions $u$ on $B(x,R)$, we have
$$\sup_{y\in B(x,r)} u(y) \leq C\cdot \inf_{y\in B(x,r)} u(y).$$
\end{theorem}

The \textit{Harnack principle at infinity} is a key principle of potential theory in hyperbolic geometry. It was established by A.~Ancona (\cite{Anc87}) in a very general framework using the concept of $\phi$-chains. It can be stated simply in the Gromov hyperbolic framework:

\begin{theorem}[Submultiplicativity of the Green function]   \label{thm:submultiplicativity Green function}
There exists a constant $C>0$ such that for all pairs of points $(x,z)\in M^2$ and all points $y\in M$ on a geodesic segment between $x$ and $z$ with $\min \{ d(x,y),d(y,z)\} \geq 1$, the Green function $G$ satisfies
$$C^{-1} \cdot G(x,y)G(y,z)\leq G(x,z)\leq C \cdot G(x,y)G(y,z).$$
\end{theorem}

We will also need another formulation of this principle. Let $\gamma$ be a geodesic ray starting at $z\in M$ and denote
$$a_i^\gamma := \gamma(4i\delta),\ i\in \N\setminus \{ 0\}$$
$$\text{and } U_i^\gamma := \{ x\in M \, | \, (x,a_i^\gamma)_z > 4i\delta-2\delta\}.$$
Let us point out that for all $i$, $a_i^\gamma \in U_i^\gamma \setminus U_{i+1}^\gamma$ (see figure \ref{fig:Harnack at infinity}).
\begin{figure}
\begin{center}
\includegraphics[width=4.5cm,height=5cm]{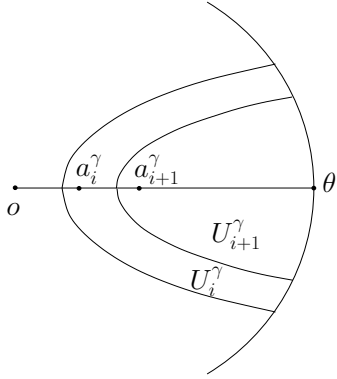}
\caption{Sets $U_i^\gamma$ and points $a_i^\gamma$}
\label{fig:Harnack at infinity}
\end{center}
\end{figure}
Let us also notice that the decreasing sequence of sets $(U_i^\gamma)$ and the sequence of points $(a_i^\gamma)$ provide a $\phi$-chain in the sense of A.~Ancona (\cite{Anc90}, page 93). Then, the Harnack principle at infinity can be stated as follow:

\begin{theorem}[\cite{Anc88}, page 12]  \label{thm:Harnack at infinity}
There exists a constant $C>0$ such that for all $\theta\in \partial M$ and for all geodesic rays $\gamma$ from $o$ to $\theta$, the following properties are satisfied:
\begin{enumerate}
 \item If $u$ and $v$ are two non-negative harmonic functions on $U_i^\gamma$, $v$ does not vanish and $u$ "vanishes" at $\overline{U_i^\gamma} \cap \partial M$, then
 $$\forall x\in U_{i+1}^\gamma, \ \frac{u(x)}{v(x)} \leq C \frac{u(a_{i+1}^\gamma)}{v(a_{i+1}^\gamma)}.$$
 \item If $u$ and $v$ are two non-negative harmonic functions on $M\setminus U_{i+1}^\gamma$, $v$ does not vanish and $u$ "vanishes" at $\partial M \setminus \overline{U_{i+1}^\gamma}$, then
 $$\forall x\in U_i^\gamma, \ \frac{u(x)}{v(x)} \leq C \frac{u(a_i^\gamma)}{v(a_i^\gamma)}.$$
\end{enumerate}
\end{theorem}


\section{Non-tangential behaviour of Brownian motion}   \label{sect:NT behaviour of Brownian motion}

In this section, we gather several lemmas ensuing Harnack inequalities. They provide key ingredients in the proofs of the main results of the paper.

\subsection{A geometric Lemma}

The following geometric lemma is one of the main tools in the ensuing proofs and in particular in the proof of Theorem \ref{thm:local Fatou}. In \cite{Mou94}, it is achieved by use of comparison theorems in pinched negative curvature. For a borelian set $E\subset \partial M$ and a real $c>0$, denote
$$\Gamma_c(E):=\bigcup_{\theta\in E} \Gamma_c^\theta.$$

\begin{lemma}  \label{lem:geometric lemma 1}
There exist $\eta>0$ and $c_0>0$ such that for all borelian sets $E\subset \partial M$ and all $c>c_0$, one has
\begin{equation*}
\forall x\not\in \Gamma_c(E), \ \mu_x(E)\leq 1-\eta.
\end{equation*}
\end{lemma}

\begin{figure}
\begin{center}
\includegraphics[width=4.5cm,height=5cm]{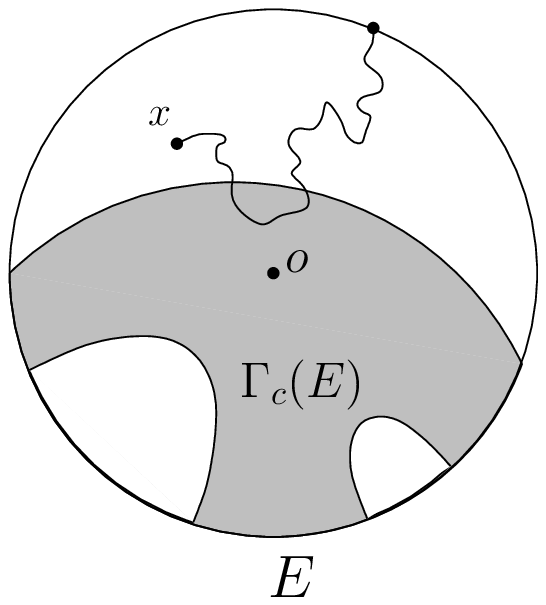}
\caption{Geometric Lemma \ref{lem:geometric lemma 1}}
\label{fig:geometric lemma 1}
\end{center}
\end{figure}
Figure \ref{fig:geometric lemma 1} illustrates Lemma \ref{lem:geometric lemma 1}. The proof follows \cite{Pet12}. We decompose it in two technical lemmas.

For a point $x\in M$, a point $\theta\in \partial M$ and a real $\alpha>0$, denote
$$A_{x,\alpha}^{\theta}= \{ \xi\in \partial M | (\xi,\theta)_x \geq \alpha\}.$$

\begin{lemma}    \label{lem:geometric Lemma 3}
There exist two constants $C_1>0$ and $d_1>0$ depending only on $\alpha$ and $\delta$ such that for all $\xi \in \partial M\setminus A_{x,\alpha}^\theta$ and all points $y$ on a geodesic ray from $x$ to $\theta$ with $d(x,y)\geq d_1$,
$$\frac{d\mu_y}{d\mu_x}(\xi)\leq C_1 \cdot G(y,x).$$
\end{lemma}

\begin{proof}
Let $\xi \in \partial M\setminus A_{x,\alpha}^\theta$. Denote by $\gamma$ a geodesic ray from $x$ to $\xi$. Choose $i$ such that $d(x,a_i^\gamma)-3\delta=4i\delta-3\delta>\alpha+4\delta$. By the hyperbolicity inequality (\ref{ineq:hyperbolicity 2}),
$$\alpha > (\xi,\theta)_x \geq \min \{ (\xi,y)_x , (y,\theta)_x \}-2\delta$$
and if $y$ lies on a geodesic ray from $x$ to $\theta$, there exists $d_2$ depending only on $\alpha$ such that $d(x,y)\geq d_2$ implies $(y,\theta)_x > \alpha+2\delta$. Thus for such a point $y$, $(\xi,y)_x\leq \alpha+2\delta$. Using once again the hyperbolicity inequality,
\begin{equation}   \label{ineq:geometric lemma 2}
\alpha+2\delta \geq (\xi,y)_x \geq \min \{ (\xi,a_i^\gamma)_x,(a_i^\gamma,y)_x \} -2\delta.
\end{equation}
We have $\xi\in \overline{U_i^\gamma}$, thus $(\xi,a_i^\gamma)_x \geq d(x,a_i^\gamma)-3\delta > \alpha+4\delta$. Combining with inequality (\ref{ineq:geometric lemma 2}), we obtain $(a_i^\gamma,y)_x\leq \alpha+4\delta$ and thus $y\not\in \overline{U_i^\gamma}$. 

\begin{figure}
\begin{center}
\includegraphics[width=4.5cm,height=5cm]{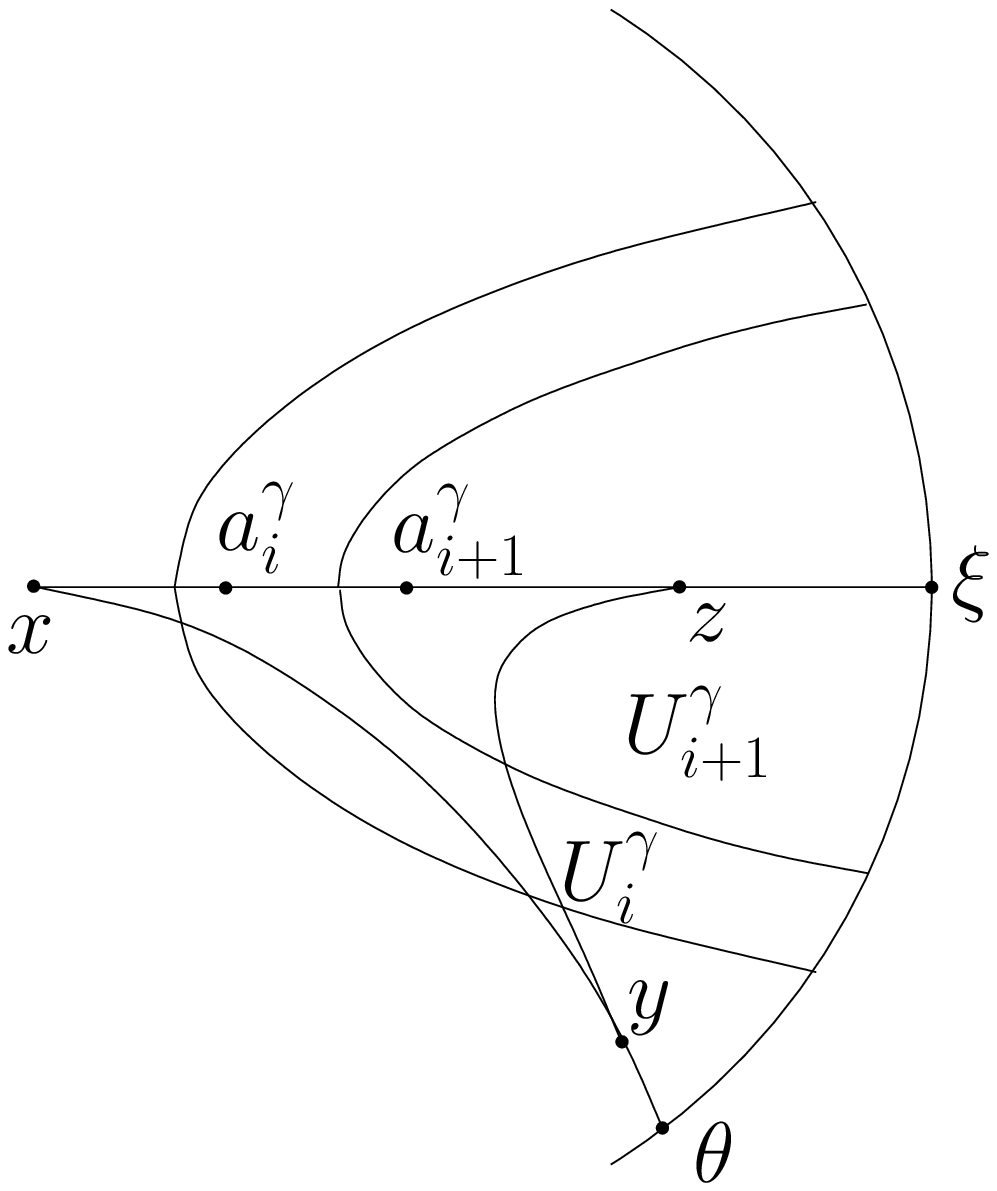}
\caption{Proof of Lemma \ref{lem:geometric Lemma 3}}
\label{fig:proof lemma 5.2}
\end{center}
\end{figure}

Let $z$ be a point of $\gamma$ in $\overline{U_{i+1}^\gamma}$. It is an exercice using hyperbolicity (see \cite{Anc90} page 85 for details) to verify that the distance between $a_i^\gamma$ and a geodesic segment between $y$ and $z$ is at most $50\delta$ (see figure \ref{fig:proof lemma 5.2}). Since $i$ is fixed, the distance between $x$ and a geodesic segment between $y$ and $z$ is bounded from above by a constant depending only on $\delta$. Thus, the submultiplicativity of the Green function on geodesic segments (Theorem \ref{thm:submultiplicativity Green function}) associated with the Harnack inequality on balls (Theorem \ref{thm:Harnack on balls}) give a constant $C_1$ depending only on $\delta$ such that $G(y,z)\leq C_1\cdot G(y,x)G(x,z)$. Since we have
$$\frac{d\mu_y}{d\mu_x}(\xi)=\lim_{z\to \xi} \frac{G(y,z)}{G(x,z)},$$
letting $z\to \xi$, $z\in \overline{U_{i+1}^\gamma}$, we obtain
$$\frac{d\mu_y}{d\mu_x}(\xi)\leq C_1 \cdot G(x,y).$$
\end{proof}

\begin{lemma}   \label{lem:geometric lemma 2}
Given $\alpha>0$, there exists a constant $\eta>0$ such that for all points $x\in M$ and all $\theta\in \partial M$,
\begin{equation*}
\mu_x\left( A_{x,\alpha}^\theta \right) \geq \eta.
\end{equation*}
\end{lemma}

\begin{proof}
Fix $\alpha>0$. We first prove that there exists $d=d(\alpha)>0$ such that for all $x\in M$, all $\theta\in \partial M$ and all points $y$ on a geodesic ray from $x$ to $\theta$ with $d(x,y)\geq d$, we have
$$\mu_y(A_{x,\alpha}^\theta)>\frac 12.$$
Note that we have
\begin{eqnarray}   \label{eq:geometric lemma 2}
\mu_y \left( \partial M \setminus A_{x,\alpha}^\theta \right) =\int_{\partial M \setminus A_{x,\alpha}^\theta} \frac{d\mu_y}{d\mu_x}(\xi) d\mu_x(\xi).
\end{eqnarray}
We deduce, from Lemma \ref{lem:geometric Lemma 3} and formula (\ref{eq:geometric lemma 2}), that for all points $y$ on a geodesic ray from $x$ to $\theta$ with $d(x,y)\geq d_1$,
$\mu_y(\partial M \setminus A_{x,\alpha}^\theta)\leq C_1 \cdot G(x,y)$.
Since the Green function $G$ has a uniform exponential decay at infinity (Proposition \ref{prop:exponential decay}), there exists a $d$ depending only on $\alpha$ and $\delta$ such that for all point $y$ on a geodesic ray from $x$ to $\theta$, with $d(x,y)\geq d$,
$$\mu_y(A_{x,\alpha}^\theta)>\frac 12.$$
We conclude by Harnack inequality on balls (Theorem \ref{thm:Harnack on balls}): if $x,y\in M$ with $d(x,y)=d$, we have
$$\mu_x(A_{x,\alpha}^\theta)\geq C(\alpha,\delta) \cdot \mu_y(A_{x,\alpha}^\theta)\geq \eta>0$$
and the lemma is proved.
\end{proof}

We can now prove Lemma \ref{lem:geometric lemma 1}.

\begin{proof}[Proof of Lemma \ref{lem:geometric lemma 1}]
Fix $c_0:=K+6\delta$, where $K$ denotes the constant coming from the roughly starlike assumption on $M$. Let $c\geq c_0$, $E$ be a borelian set in $\partial M$, and $x\not\in \Gamma_c(E)$. Choose a geodesic ray $\overline\gamma$ with origin $o$ such that $d(x,\overline\gamma)\leq K$, and denote by $\theta\in \partial M$ the endpoint of $\overline\gamma$. Since $x\not\in \Gamma_c(E)$, $\theta \not\in E$ (see figure \ref{fig:proof lemma 1}). 
\begin{figure}
\begin{center}
\includegraphics[width=4.5cm,height=5cm]{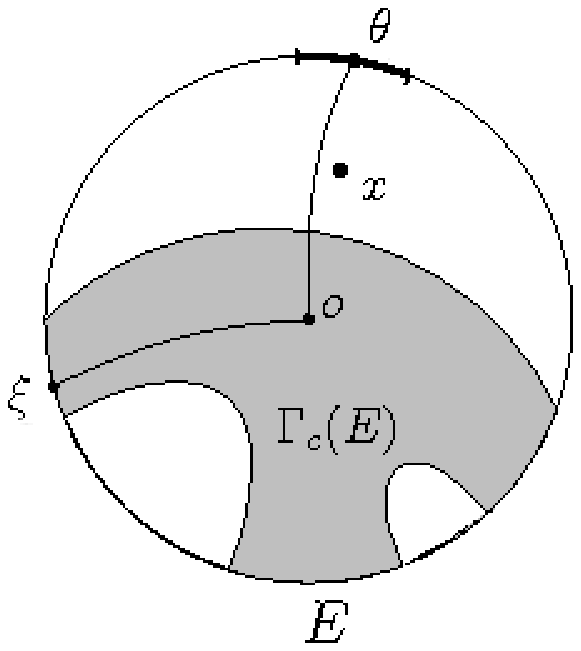}
\caption{Proof of Lemma \ref{lem:geometric lemma 1}}
\label{fig:proof lemma 1}
\end{center}
\end{figure}
We prove that there exists a constant $\alpha>0$ depending only on $\delta$ and $K$ such that $A_{x,\alpha}^\theta \subset \partial M \setminus E$. We want to bound the quantity $(\xi,\theta)_x$ uniformly from above for $\xi\in E$. Fix $\xi\in E$. Inequality (\ref{ineq:hyperbolicity 2}) gives
$$\min \{ (\xi,\theta)_x,(\xi,o)_x \} \leq (\theta,o)_x +2\delta.$$
On one hand, since by inequality (\ref{ineq:Gromov product and geodesics}), $(\theta,o)_x\leq d(x,\overline\gamma)+2\delta \leq K+2\delta$, we have $\min \{ (\theta,\theta)_x,(\xi,o)_x \} \leq K+4\delta$. On the other hand, denoting by $\gamma$ a geodesic ray from $o$ to $\xi$, we have $(\xi,o)_x\geq d(x,\gamma)-2\delta\geq c-2\delta\geq K+4\delta$. We thus deduce that $(\xi,\theta)_x\leq K+4\delta$. Since this holds for all $\xi\in E$, we obtain
$$A_{x,K+5\delta}^\theta \cap E =\emptyset.$$
By Lemma \ref{lem:geometric lemma 2}, there exists an $\eta>0$ depending only on $\delta$ and $K$ such that 
$$\mu_x(E)\leq 1-\eta,$$
which concludes the proof of Lemma \ref{lem:geometric lemma 1}.
\end{proof}

For a borelian set $E\subset \partial M$, denote by $v_E(x):=\mu_x(E)=\mathbb P_x(X_\infty \in E)$.
Let $U$ be an open set in $M$. Recall that a point $\theta\in \partial M$ is called \textit{tangential for $U$} if for all $c>0$, the set $\Gamma_c^\theta \setminus U$ is bounded.
 The following corollary of Lemma \ref{lem:geometric lemma 1} asserts that for almost every point $\theta$ that is tangential for an open set $U$ in $M$, Brownian motion "ends its life $\mathbb P_o^\theta$-almost surely" in $U$.

\begin{corollary} \label{cor:end lemma}
Let $U$ be an open set in $M$. Then for $\mu$-almost all $\theta$ that are tangential for $U$, $\mathbb P_o^\theta$-almost surely, $X_t\in U$ for $t$ large enough.
\end{corollary}

\begin{proof}
Let $c>c_0$, where $c_0$ is the constant given in Lemma \ref{lem:geometric lemma 1}. Denote by $T$ the set tangential points for $U$ and, for $N\in \N$, let
$$T_N:= \{ \theta\in \partial M \, | \, \Gamma_c^\theta \setminus U \subset B(o,N) \}.$$
By countable union, it is sufficient to prove, for each $N\in \N$, that for $\mu$-almost all $\theta\in T_N$, $X_t \in U$ for $t$ large enough. Fix $N\in \N$.
On one hand, since $v_{T_N}$ is a bounded harmonic function, Lemma \ref{lem:bounded harmonic functions}  asserts that for $\mu$-almost all $\theta\in T_N$, $\mathbb P_o^\theta$-almost surely,
$$\lim_{t\to\infty} v_{T_N}(X_t)= \ind{T_N}(\theta).$$
On the other hand, by Lemma \ref{lem:geometric lemma 1}, 
$$\forall x\not\in \Gamma_c(T_N), v_{T_N}(x)\leq 1-\eta.$$
Thus, for $\mu$-almost all $\theta\in T_N$, $\mathbb P_o^\theta$-almost surely, $X_t\in \Gamma_c(T_N)$ for $t$ large enough. Notice that for such a point $\theta\in T_N$ Brownian motion leaves $\mathbb P_o^\theta$-almost surely the ball $B(o,N)$ and that $\Gamma_c(T_N)\setminus B(o,N) \subset U$ by definition of $T_N$. This proves the corollary.

\end{proof}

\begin{corollary}  \label{cor:spikes lemma}
Let $c>c_0$ and $E$ be a Borelian subset of $\partial M$. Every $\theta\in\partial M$ such that $v_E$ converges non-tangentially to 1 at $\theta$ is tangential for $\Gamma_c(E)$. In particular, $\mu$-almost all $\theta\in E$ is tangential for $\Gamma_c(E)$.
\end{corollary}

\begin{proof}
Let $\theta\in \partial M$ be such that $v_E$ converges non-tangentially to 1 at $\theta$ and let $\Gamma_e^\theta$ be a non-tangential tube with vertex $\theta$. Assume that $\Gamma_e^\theta \setminus \Gamma_c(E)$ is not bounded. Then there exists a sequence $(x_k)_k$ of points in $\Gamma_e^\theta \setminus \Gamma_c(E)$ such that $d(o,x_k)>k$. We thus have $v_E(x_k)\to 1$. However, by Lemma \ref{lem:geometric lemma 1}, $v_E(x_k)\leq 1-\eta$, which gives a contradiction and proves the main statement of the corollary. In addition, by Lemma \ref{lem:bounded harmonic functions}, $v_E$ converges non-tangentially to 1 at $\mu$-almost all $\theta\in E$ and $\mu$-almost every $\theta\in E$ is tangential for $\Gamma_c(E)$.
\end{proof}

\subsection{Behaviour of Green functions}

The next Lemma yields an estimate for the increasing rate of the minimal harmonic function $K(\cdot,\theta)$ along non-tangential tubes with vertex $\theta$. The proof is a straightforward application of Theorem \ref{thm:submultiplicativity Green function} and Theorem \ref{thm:Harnack on balls}  (see \cite{Anc90} page 99).

\begin{lemma}    \label{lem:comparison Green Poisson}
For all $c>0$, there exist $0<C<1$ and $R>0$ such that for all $\theta\in \partial M$, all $x\in \Gamma_c^\theta$, and all $y\in \Gamma_c^\theta \setminus B(o,R)$,
$$C \leq G(o,x)K(x,\theta) \text{ and } G(o,y)K(y,\theta)\leq C^{-1}.$$
\end{lemma}

For an open set $U\subset M$, denote by $G_U$ the Green function of $U$. The next lemma allows to compare $G$ and $G_U$ for a class of subsets $U\subset M$. This will be useful in the proof of Theorem \ref{thm:density energy}.

\begin{lemma}   \label{lem:comparison Green functions}
Fix large $c>e>0$ and $\theta\in \partial M$. Let $U$ be an open subset of $M$ containing $\Gamma_c^\theta$ and denote by $\tau$ the exit time of $U$. Then we have
$$\lim_{x\to\theta, x\in \Gamma_e^\theta} \frac{G_U(o,x)}{G(o,x)}=\mathbb P_o^\theta(\tau=+\infty).$$
\end{lemma}

\begin{proof}
Since $G(\cdot,x)$ vanishes at infinity,
\begin{eqnarray*}
G_U(o,x) &=& G(o,x)-\mathbb E_o \left[ G(X_\tau,x) \right] \\
&=& G(o,x) \left( 1-\mathbb E_o \left[ \frac{G(X_\tau,x)}{G(o,x)}\cdot \ind{\tau<+\infty} \right] \right).
\end{eqnarray*}
Recall that for $\tau<+\infty$, $\lim_{x\to\theta} \frac{G(X_\tau,x)}{G(o,x)}=K(X_\tau,\theta)$. Hence, provided changing the order of limit and expectation is justified, we have
$$\lim_{x\in \Gamma_e^\theta, x\to\theta} \mathbb E_o \left[ \frac{G(X_\tau,x)}{G(o,x)}\cdot \ind{\tau<+\infty} \right] =\mathbb E_o \left[ K(X_\tau,\theta)\cdot \ind{\tau<+\infty}\right]=\mathbb P_o^\theta (\tau<+\infty)$$
and the lemma follows. It remains to justify changing the order of limit and expectation, which will be achieved by proving the following property:

There exists a constant $C>0$ such that
\begin{equation}   \label{inversion limit expectation}
\forall x\in \Gamma_e^\theta \setminus B(o,c), \forall z\not\in \Gamma_c^\theta, \frac{G(z,x)}{G(o,x)} \leq C\cdot K(z,\theta).
\end{equation}

To prove (\ref{inversion limit expectation}), we will apply Theorem \ref{thm:Harnack at infinity} several times with $u=G(\cdot,y)$ for a point $y\in M$ and $v=K(\cdot,\theta)$. The function $G(\cdot,y)$ is positive harmonic on $M\setminus \{ y\}$, vanishes at infinity, and the function $K(\cdot,\theta)$ is positive harmonic. The assumptions of Theorem \ref{thm:Harnack at infinity} will thus always be satisfied. In the rest of the proof, the constants depend only on the Gromov hyperbolicity constant $\delta$, on the roughly starlike constant $K$ and on $c$ and $e$. 

Let $z\not\in \Gamma_c^\theta$, $x\in \Gamma_e^\theta\setminus B(o,c)$ and let $\gamma$ be a geodesic ray starting at $o$ and converging to $\theta$ such that $d(x,\gamma)<e$. First, remark that by Theorem \ref{thm:Harnack on balls}, we can assume $x\in \gamma\setminus B(o,c)$. Indeed, if $x'\in \gamma\setminus B(o,c)$ is such that $d(x,\gamma)=d(x,x')$, then $\frac{G(z,x)}{G(o,x)}\leq C_0 \cdot \frac{G(z,x')}{G(o,x')}$.

Denote, for $i\in \N^*, a_i:=a_i^\gamma=\gamma(4i\delta)$ and $U_i:=U_i^\gamma=\{ y\in M \, | \, (y,a_i)_o >d(o,a_i)-2\delta\}$. We split the proof in different cases:

\textbf{Case 1:} $z\not\in U_3$.

By Theorem \ref{thm:Harnack at infinity}, there exists $C_1>0$ such that
$$\frac{G(z,x)}{K(z,\theta)}\leq C_1\cdot \frac{G(a_2,x)}{K(a_2,\theta)}.$$
By definition of $a_2$, $d(o,a_2)=8\delta$ and using once again Theorem \ref{thm:Harnack on balls}, there exists $C_2>0$ such that
$$\frac{G(z,x)}{K(z,\theta)}\leq  C_2 \cdot \frac{G(o,x)}{K(o,\theta)}=C_2\cdot G(o,x),$$
which gives (\ref{inversion limit expectation}) in case 1.

\textbf{Case 2:} $z\in U_3$. 

By definition of $U_3$, $d(o,z)>d(a_3,z)+8\delta$. Denote by $o'$ a point in $\gamma$ such that $d(z,o')=\min_{z'\in \gamma} d(z,z')$. Since $d(a_3,z)\geq d(z,o')$, we have $d(o,o')\geq  8\delta$. Denote by $\gamma'$ a geodesic ray starting at $o'$ and within a distance at most $K$ from $z$ (recall that $M$ is $K$-roughly starlike). If $c$ is large enough (depending on $\delta$ and $K$), it is an easy exercise to prove that $z\in U_3^{\gamma'}$. We can thus apply Theorem \ref{thm:Harnack at infinity} to have
$$\frac{G(z,x)}{K(z,\theta)}\leq C_3\cdot \frac{G(\gamma'(8\delta),x)}{K(\gamma'(8\delta),\theta)}.$$
Hence it is sufficient to prove (\ref{inversion limit expectation}) for a point $z$ within distance at most $8\delta$ from $\gamma$. Let $z$ be such a point and denote again by $o'\in\gamma$ a point so that $d(z,o')=\min_{z'\in\gamma} d(z,z')$. There are two cases, illustrated by Figure \ref{fig:utilisation_Harnack}.

\begin{figure}
\begin{center}
\includegraphics[width=12cm,height=4cm]{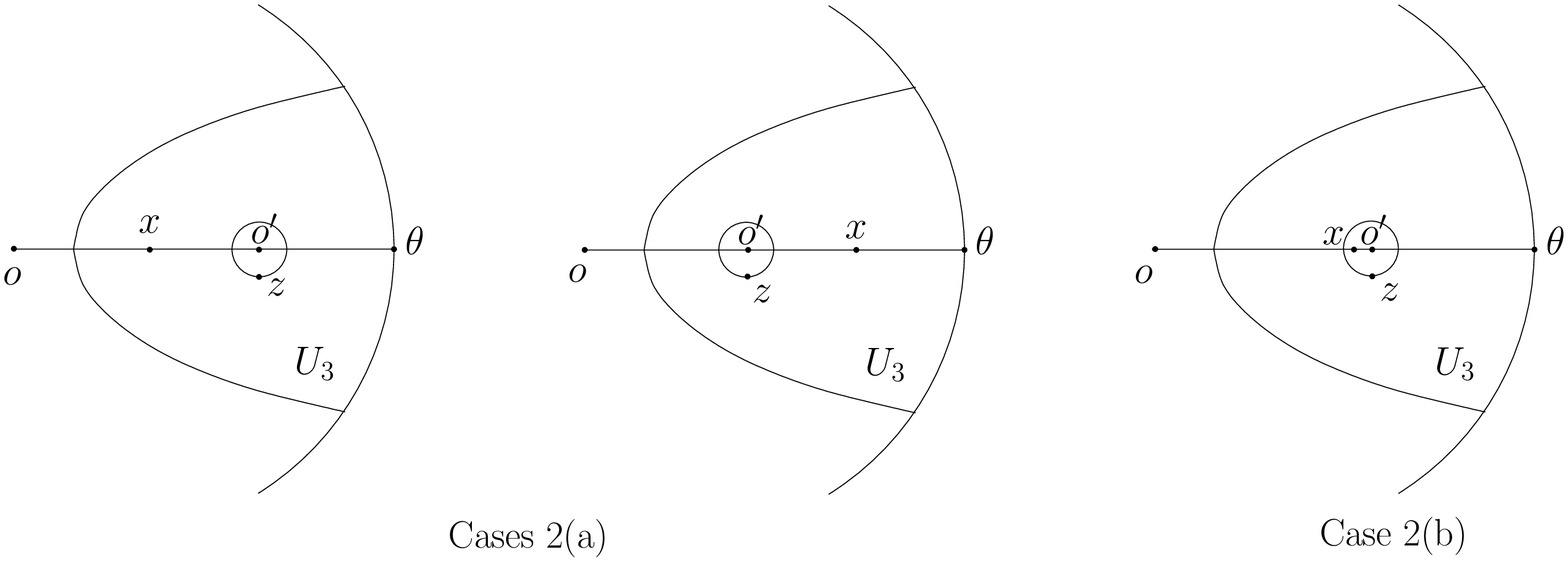}
\caption{Proof of Lemma \ref{lem:comparison Green functions}}
\label{fig:utilisation_Harnack}
\end{center}
\end{figure}

\textbf{Case 2(a):} $d(x,o')>16\delta$.

In that case, Theorem \ref{thm:Harnack on balls} yields that it is sufficient to prove
\begin{equation}  \label{eq:case 2(a)}
\frac{G(o',x)}{G(o,x)}\leq C_4\cdot K(o',\theta).
\end{equation}

Recall that the three points $o$, $o'$ and $x$ lie on the geodesic ray $\gamma$.
\begin{itemize}
 \item If $o'$ is between $o$ and $x$, we apply Theorem \ref{thm:Harnack at infinity} with base point $x$ and with $a_i, i=1,2$ the points of $\gamma$ such that $d(o,x)=d(o,a_i)+4i\delta$ and get
$$\frac{G(o',x)}{K(o',\theta)}\leq C_5\cdot \frac{G(a_1,x)}{K(a_1,\theta)}.$$
Using Theorem \ref{thm:Harnack on balls}, we obtain (\ref{eq:case 2(a)}).\\

\item If $x$ is between $o$ and $o'$, we apply Theorem \ref{thm:submultiplicativity Green function} (recall that $K(\cdot,\theta)=\lim_{y\to\theta} \frac{G(\cdot,y)}{G(o,y)}$) and obtain
$$\frac{K(x,\theta)}{K(o',\theta)}\leq C_6\cdot G(x,o').$$
Since $G(o,x)K(x,\theta)\geq C$ (Lemma \ref{lem:comparison Green Poisson}) and since $G(x,o')\leq C_7$ (Proposition \ref{prop:exponential decay}), we obtain (\ref{eq:case 2(a)}).
\end{itemize}

\textbf{Case 2(b):} $d(x,o')\leq 16\delta$.

Since $8\delta\leq d(x,z)\leq 24\delta$, $K(x,\theta)\leq C_{8} \cdot K(z,\theta)$ and $G(x,z)\leq C_{9}$. Combining these two inequalities with $G(o,x)K(x,\theta)\geq C$, we get property (\ref{inversion limit expectation}) in case 2(b). Changing the order of limit and expectation is justified and the proof is complete.

\end{proof}

\subsection{Brownian motion and non-tangential sets}

Harnack principles allow to prove the following lemma (see Figure \ref{fig:BMandBalls} for an illustration), which helps build connections between stochastic properties and non-tangential properties. A.~Ancona stated it in a potential theory terminology (\cite{Anc90}, Lemma 6.4) and used it to prove a Fatou's Theorem.

\begin{lemma}  \label{lem:Brownian motion and NT balls}
Consider a sequence of balls of fixed positive radius whose centers converge non-tangentially to a point $\theta\in \partial M$, that is, converge to $\theta$ staying in a non-tangential cone $\Gamma_c^\theta$ for some $c>0$. Then Brownian motion meets $\mathbb P_o^\theta$-almost surely infinitely many of these balls.
\end{lemma}

\begin{figure}
\begin{center}
\includegraphics[width=6cm,height=5cm]{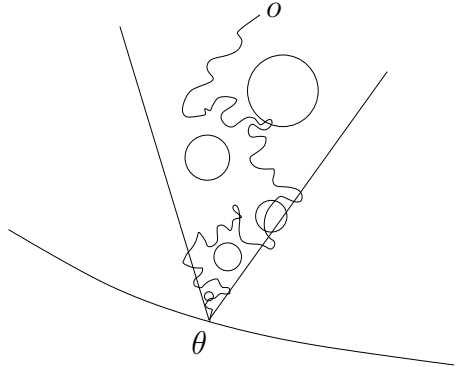}
\caption{Brownian motion and non-tangential balls}
\label{fig:BMandBalls}
\end{center}
\end{figure}

We end this section by proving the following lemma:

\begin{lemma}  \label{lem:proba tau infinite}
Let $U$ be a connected open subset of $M$ such that $o\in U$, and let $\tau$ denote the exit time of $U$. For every $\theta\in \partial M$ such that $\mathbb P_o^\theta$-almost surely, $X_t\in U$ for $t$ large enough, we have $\mathbb P_o^\theta(\tau=+\infty)>0$.
\end{lemma}

\begin{remark}
By Corollary \ref{cor:end lemma}, the conclusion holds in particular at $\mu$-almost every point $\theta$ tangential for $U$.
\end{remark}

\begin{proof}[Proof of Lemma \ref{lem:proba tau infinite}]
Let $\theta\in \partial M$ be such that $\mathbb P_o^\theta$-almost surely, $X_t\in U$ for $t$ large enough. Denote by $h$ the non-negative harmonic function on $U$ defined by $h(x):= K(x,\theta) \mathbb P_x^\theta(\tau=+\infty)$. By the maximum principle, $h$ is either positive, or identically zero. We have
$$1=\lim_{N \to \infty} \mathbb P_o^\theta(\forall t\geq \tau_N,\, X_t\in U),$$
where $\tau_N$ denotes the exit time of $B(o,N)$.
Let $N$ be large enough so that $\mathbb P_o^\theta(\forall t\geq \tau_N,\, X_t\in U)>0$. By the strong Markov property,
$$\mathbb P_o^\theta(\forall t\geq \tau_N,\, X_t\in U)=\mathbb E_o^\theta \left[ \mathbb P_o^\theta(\forall t\geq \tau_N,\, X_t\in U) |  \tau_N \right]=\mathbb E_o^\theta [\varphi(X_{\tau_N})],$$
where $\varphi(x):= \mathbb P_x^\theta(\tau=+\infty)$ if $x\in U$ and $\varphi(x):=0$ otherwise. The function $\varphi$, and therefore $h$, is not identically zero. The function $h$ is thus positive and $\mathbb P_o^\theta(\tau=+\infty)>0$, which proves the Lemma.
\end{proof}


\section{Local Fatou theorem}     \label{sect:local Fatou theorem}

The aim of this section is to prove Theorem \ref{thm:local Fatou}. The proof is similar to the proof of Theorem 2 in \cite{Mou07}, and based upon the use of Lemma \ref{lem:proba tau infinite}, which is achieved using Lemma \ref{lem:geometric lemma 1}. Although the main difference with \cite{Mou07} lies in Lemma \ref{lem:geometric lemma 1}, we give here a detailed proof.

\begin{proof}[Proof of Theorem \ref{thm:local Fatou}]

We can assume, without loss of generality, that $U$ is connected (since $U$ is open, it has a countable number of connected components) and that $o\in U$. Denote again by $\tau$ the exit time of $U$.  Let $u$ be a non-negative harmonic function on $U$. The martingale property asserts that $(u(X_{t\wedge \tau}))$ is a non-negative local martingale and therefore converges $\mathbb P_o$-almost surely. By formula (\ref{eq:conditioning}), for $\mu$-almost all $\theta\in \partial M$, $(u(X_{t\wedge \tau}))$ converges $\mathbb P_o^\theta$-almost surely.

By Lemma \ref{lem:proba tau infinite}, for $\mu$-almost all $\theta$ that is tangential for $U$, we have
$$\mathbb P_o^\theta (\tau=+\infty \text{ and } (u(X_t)) \text{ converges})>0.$$
Let $\theta$ be such a point. Denoting by $\tilde u(x)=u(x)$ for $x\in U$ and $\tilde u(x)=0$ otherwise, the asymptotic zero-one law asserts that $\tilde u$ converges stochastically at $\theta$. Denote by $\ell$ the stochastic limit of $\tilde u$ at $\theta$ and assume that $\tilde u$ (and therefore $u$) does not converge non-tangentially to $\ell$ at $\theta$. We will obtain a contradiction with Lemma \ref{lem:Brownian motion and NT balls}. These step is standard (see for instance \cite{BD63} page 403 and \cite{Anc90} page 100). There exist $c>0$, $\varepsilon>0$ and a sequence $(y_k)_k$ of points in $\Gamma_c^\theta \setminus B(o,R)$ converging to $\theta$ such that for every $k$, $|u(y_k)-\ell |>2\varepsilon$, where $R>0$ is such that $\Gamma_{c+1}^\theta \setminus B(o,R) \subset U$. By Harnack inequalities, we have, even replacing $2\varepsilon$ by $\varepsilon$, the same inequality on $B(y_k,\lambda)$ for a $0<\lambda<1$ independent of $k$. By Lemma \ref{lem:Brownian motion and NT balls}, Brownian motion meets $\mathbb P_o^\theta$-almost surely infinitely many of the balls $B(y_k,\lambda)$. Let $\omega$ be a generic trajectory such that $(X_t(\omega))_t$ meets infinitely many of these balls, $\tau(\omega)=+\infty$ and $\lim_{t\to +\infty} u(X_t(\omega))=\ell$. There exists $t_0$ such that for all $t\geq t_0$, $|u(X_t(\omega))-\ell | \leq \varepsilon$. By compactness, $(X_t(\omega))_{t\geq t_0}$ meets at least one of the balls $B(y_k,\lambda)$, that is there exists $t_1\geq t_0$ such that $X_{t_1}(\omega)\in B(y_k,\lambda)$ for some $k$. Then
$$0<\varepsilon < |u(X_{t_1}(\omega))-\ell | \leq \varepsilon,$$
which yields a contradiction.
The theorem is proved.
\end{proof}

We end this section by proving Corollary \ref{cor:bnd from below}.
\begin{proof}[Proof of Corollary \ref{cor:bnd from below}]
Let $u$ be a harmonic function on $M$. We have to prove that $u$ converges non-tangentially at $\mu$-almost all points $\theta\in \partial M$ where it is non-tangentially bounded from below. Fix $c>c_0$ (where $c_0$ comes from Lemma \ref{lem:geometric lemma 1}) and for $m\in\N$, let
$$A_c^m:= \{ \theta\in \partial M \, | \, \forall x\in \Gamma_c^\theta, u(x)\geq -m \}.$$
It is sufficient to prove that for every $m\in \N$, $u$ converges non-tangentially at $\mu$-almost all $\theta\in A_c^m$. Let $m\in \N$ and $U:= \Gamma_c(A_c^m)$. The function $u+m$ is non-negative harmonic on $U$. By Theorem \ref{thm:local Fatou}, it converges non-tangentially at $\mu$-almost all points $\theta$ tangential for $U$ and so the same holds for the function $u$. By corollary \ref{cor:spikes lemma}, $\mu$-almost all $\theta\in A_c^m$ is tangential for $U$ and the proof is complete.
\end{proof}


\section{Density of energy}    \label{sect:density of energy}

In this section, we prove Theorem \ref{thm:density energy} and Corollary \ref{thm:energy}. 
Let us define, for $u$ harmonic on $M$, $\theta\in \partial M$ and $c>0$ the \textit{density of energy}
$$D_c^r(\theta):=-\frac 12 \int_{\Gamma_c^\theta} \Delta |u-r|(dx).$$
We refer to \cite{Bro88,Mou07} for introductions to the density of area integral and to the density of energy, respectively. Notice that by Sard's Theorem, for almost all $r\in \R$, $D_c^r(\theta)=\int_{\Gamma_c^\theta} | \nabla u(x)] \sigma_r(dx)$, where $\sigma_r$ is the hypersurface measure on $\{ u=r\}$. In addition, by the coarea formula, the \textit{non-tangential energy} equals
$$J_c^\theta:=\int_{\Gamma_c^\theta} | \nabla u |^2 d\nu_M=\int_{r\in \R} D_c^r(\theta) dr.$$

\begin{proof}[Proof of Theorem \ref{thm:density energy}]

In order to prove Theorem \ref{thm:density energy}, we have to prove that for all $c>0$:
\begin{description}
\item[Step 1] $u$ converges non-tangentially at $\mu$-almost all $\theta\in \partial M$ where $D_c^0(\theta)<+\infty$.
 \item[Step 2] $\sup_{r\in \R} D_c^r(\theta)<+\infty$ for $\mu$-almost all $\theta\in \partial M$ where $u$ converges non-tangentially;
\end{description} 

\textbf{Step 1:} the proof goes as in the main Theorem of \cite{Mou07}, proved in the framework of manifold of pinched negative curvature. Thus, we give only the main ideas of the proof.
The proof is based upon Theorem \ref{thm:local Fatou}, Lemma \ref{lem:geometric lemma 1} and Lemma \ref{lem:comparison Green Poisson}.

For $m\in \N$, denote
$$\mathcal D_c^m:= \{ \theta\in \partial M \, | \, D_c^0(\theta)\leq m\}$$
and $\Gamma:= \Gamma_c(\mathcal D_c^m)$.
It is sufficient to prove that for all $m\in \N$, $u$ converges non-tangentially at $\mu$-a.e. $\theta\in \mathcal D_c^m$. Fix $m\in \N$ and recall that $v_{\mathcal D_c^m}(x)=\mathbb P(X_\infty \in \mathcal D_c^m)$. First we prove, using Lemmas \ref{lem:geometric lemma 1} and \ref{lem:comparison Green Poisson} that there exists $\alpha\in (0,1)$ such that $\{v_{\mathcal D_c^m} \geq \alpha\}\subset \Gamma$ and
$$I:=-\int_{\{ v_{\mathcal D_c^m} \geq \alpha \}} G(o,x) \Delta |u|(dx)<+\infty.$$

Then we prove that for an increasing sequence of compact regular domains $V_n$ such that $\bigcup_n V_n=\{ v_{\mathcal D_c^m} \geq \alpha \}$,
$$\sup_n \mathbb E_o[ |u(X_{\tau_n})|] \leq |u(o)|+I,$$
where $\tau_n$ is the exit time of $V_n$.  This allows us to decompose $u$ as the difference of two non-negative harmonic functions on $\{ v_{\mathcal D_c^m} \geq \alpha \}$ (see \cite{Bro88}). Applying Theorem \ref{thm:local Fatou} to both functions, we get that $u$ converges non-tangentially at $\mu$-almost all tangential $\theta$ for $\{ v_{\mathcal D_c^m} \geq \alpha \}$. By Lemma \ref{lem:bounded harmonic functions}, $v_{\mathcal D_c^m}$ converges non-tangentially to 1 at $\mu$-almost all $\theta\in \mathcal D_c^m$. Such a $\theta$ is thus tangential for $\{ v_{\mathcal D_c^m} \geq \alpha \}$. Hence $u$ converges non-tangentially at $\mu$-almost all $\theta\in \mathcal D_c^m$ and the proof of Step 1 is complete.

\textbf{Step 2:} For $m\in \N$ and $c>0$, denote 
$$\majN_c^m:=\{ \theta\in \partial M \, | \, \sup_{x\in \Gamma_c^\theta} |u(x)|\leq m\}.$$
It is sufficient to show that for all $m\in \N$ and all $c>e>0$, $\sup_{r\in \R} D_e^r(\theta)<+\infty$ for $\mu$-a.e. $\theta\in \majN_c^m$. Fix $c>e>0$ and $m\in \N$. Let $\Gamma:=\Gamma_c(\majN_c^m)$ and let $\tau$ be the exit time of $\Gamma$. Let $\Gamma_n$ be an increasing sequence of bounded domains such that $\bigcup_n \Gamma_n=\Gamma$ and let $\tau_n$ be the exit time of $\Gamma_n$. The local martingale $(u(X_{t\wedge \tau}))$ is bounded by $m$ and thus by Barlow-Yor inequalities (\cite{BY81}), $\mathbb E_o\left[ \sup_{r\in \R} L_{\tau}^r \right] <+\infty$, where $L_t^r$ denotes the local time in $r$ of the local martingale $(u(X_t))$. Formula (\ref{eq:conditioning}) gives that for $\mu$-almost every $\theta\in \partial M$, $\mathbb E_o^\theta [\sup_{r\in\R} L_{\tau}^r]<+\infty$ and in particular, $\sup_n \mathbb E_o^\theta [\sup_{r\in\R} L_{\tau_n}^r]<+\infty$. 

We now use the following Lemma, whose proof works exactly as Proposition 2 in \cite{Bro88}. 

\begin{lemma}   \label{lem:local time}
Let $u$ be a harmonic function on $M$, $r\in \R$, and let $L_t^r$ denote the local time in $r$ of the local martingale $(u(X_t))$. Let also $U$ be a bounded domain in $M$ and $\tau$ be the exit time of $U$. We have
$$\mathbb E_o^\theta [L_\tau^r]=-\int_U G_U(o,x) K(x,\theta) \Delta |u-r|(dx).$$
\end{lemma}

By Lemma \ref{lem:local time}, for $\mu$-almost every $\theta\in \partial M$, 
$$\sup_n \sup_{r\in R} -\int_{\Gamma_n} G_{\Gamma_n}(o,x) K(x,\theta) \Delta |u-r|(dx)<+\infty.$$ Since for every $n\in \N$, $G_{\Gamma_n}(o,x) \ind{\Gamma_n}(x)\leq G_{\Gamma_{n+1}}(o,x) \ind{\Gamma_{n+1}}(x)$, by the monotone convergence theorem, we have for $\mu$-almost all $\theta\in\partial M$,
\begin{equation}   \label{eq:local time expectation}
\sup_{r\in \R}\ -\int_{\Gamma} G_{\Gamma}(o,x) K(x,\theta) \Delta |u-r|(dx)<+\infty.
\end{equation}

On the other hand, by Lemma \ref{lem:proba tau infinite}, for $\mu$-almost every $\theta\in \majN_c^m$, $\mathbb P_o^\theta(\tau=+\infty)>0$. Hence by Lemmas \ref{lem:comparison Green Poisson} and \ref{lem:comparison Green functions}, for $\mu$-almost every $\theta\in \majN_c^\theta$, there exist $R>0$ and $C>0$ such that 
\begin{equation}
\forall x\in \Gamma_e^\theta \setminus B(o,R), \, G_\Gamma(o,x)K(x,\theta)\geq C.
\end{equation}
Combining with (\ref{eq:local time expectation}), we obtain that for $\mu$-almost every $\theta\in \majN_c^m$,
$$2 \sup_{r\in \R} D_e^r(\theta)=\sup_{r\in\R}\ -\int_{\Gamma_e^\theta} \Delta |u-r|(dx)<+\infty$$
and Theorem \ref{thm:density energy} is proved.
\end{proof}

\begin{proof}[Proof of Corollary \ref{thm:energy}]
Recall that the non-tangential energy at $\theta\in \partial M$ is
$$J_c^\theta =\int_{r\in\R} D_c^r(\theta) dr.$$
Note that if $u$ converges non-tangentially at $\theta\in \partial M$, $D_c^r(\theta)=0$ for $|r|$ large enough and therefore Theorem \ref{thm:density energy} implies that $u$ has finite non-tangential energy at $\mu$-almost all $\theta\in\partial M$ where $u$ converges non-tangentially. 

For $m\in \N$ and $c>0$, denote
$$\majJ_c^m:= \left\{ \theta\in\partial M \, | \, \int_{\Gamma_c^\theta} |\nabla u|^2 d\nu_M \leq m\right\}.$$
It is sufficient to prove that for all $m\in\N$, $u$ converges non-tangentially at $\mu$-almost all $\theta\in \majJ_c^m$. We have
$$\int_\R \int_{\majJ_c^m} D_c^r(\theta) d\mu_o(\theta) dr= \int_{\majJ_c^m} J_c^\theta d\mu_o(\theta) \leq m.$$
Then for almost every $r\in \R$, we have $\int_{\majJ_c^m} D_c^r(\theta) d\mu_o(\theta)<+\infty$. For such a real $r\in \R$, we have thus that for $\mu$-almost all $\theta\in \majJ_c^m$, $D_c^r(\theta)<+\infty$ and by Theorem \ref{thm:density energy}, for $\mu$-almost all $\theta\in\majJ_c^m$, $u$ converges non-tangentially at $\theta$. 

\end{proof}

\bibliography{biblio}

\begin{thebibliography}{GDLH90}

\bibitem[Anc87]{Anc87}
A.~Ancona.
\newblock Negatively curved manifolds, elliptic operators, and the {M}artin
  boundary.
\newblock {\em Ann. of Math. (2)}, 125(3):495--536, 1987.

\bibitem[Anc88]{Anc88}
A.~Ancona.
\newblock Positive harmonic functions and hyperbolicity.
\newblock In {\em Potential theory---surveys and problems ({P}rague, 1987)},
  volume 1344 of {\em Lecture Notes in Math.}, pages 1--23. Springer, Berlin,
  1988.

\bibitem[Anc90]{Anc90}
A.~Ancona.
\newblock Th\'eorie du potentiel sur les graphes et les vari\'et\'es.
\newblock In {\em \'{E}cole d'\'et\'e de {P}robabilit\'es de {S}aint-{F}lour
  {XVIII}---1988}, volume 1427 of {\em Lecture Notes in Math.}, pages 1--112.
  Springer, Berlin, 1990.

\bibitem[AP08]{AP08}
L.~Atanasi and M.A. Picardello.
\newblock The {L}usin area function and local admissible convergence of
  harmonic functions on homogeneous trees.
\newblock {\em Trans. Amer. Math. Soc.}, 360(6):3327--3343, 2008.

\bibitem[AS85]{AS85}
M.T. Anderson and R.~Schoen.
\newblock Positive harmonic functions on complete manifolds of negative
  curvature.
\newblock {\em Ann. of Math. (2)}, 121(3):429--461, 1985.

\bibitem[BD63]{BD63}
M.~Brelot and J.~L. Doob.
\newblock Limites angulaires et limites fines.
\newblock {\em Ann. Inst. Fourier (Grenoble)}, 13(fasc. 2):395--415, 1963.

\bibitem[BH99]{BH99}
M.R. Bridson and A.~Haefliger.
\newblock {\em Metric spaces of non-positive curvature}, volume 319 of {\em
  Grundlehren der Mathematischen Wissenschaften [Fundamental Principles of
  Mathematical Sciences]}.
\newblock Springer-Verlag, Berlin, 1999.

\bibitem[BHK01]{BHK01}
M.~Bonk, J.~Heinonen, and P.~Koskela.
\newblock Uniformizing {G}romov hyperbolic spaces.
\newblock {\em Ast\'erisque}, (270):viii+99, 2001.

\bibitem[Bro78]{Bro78}
J.~Brossard.
\newblock Comportement ``non-tangentiel'' et comportement ``brownien'' des
  fonctions harmoniques dans un demi-espace. {D}\'emonstration probabiliste
  d'un th\'eor\`eme de {C}alderon et {S}tein.
\newblock In {\em S\'eminaire de {P}robabilit\'es, {XII} ({U}niv. {S}trasbourg,
  {S}trasbourg, 1976/1977)}, volume 649 of {\em Lecture Notes in Math.}, pages
  378--397. Springer, Berlin, 1978.

\bibitem[Bro88]{Bro88}
J.~Brossard.
\newblock Densit\'e de l'int\'egrale d'aire dans {${\bf R}_+^{n+1}$} et limites
  non tangentielles.
\newblock {\em Invent. Math.}, 93(2):297--308, 1988.

\bibitem[Bus82]{Bus82}
P.~Buser.
\newblock A note on the isoperimetric constant.
\newblock {\em Ann. Sci. \'Ecole Norm. Sup. (4)}, 15(2):213--230, 1982.

\bibitem[BY81]{BY81}
M.~T. Barlow and M.~Yor.
\newblock ({S}emi-) martingale inequalities and local times.
\newblock {\em Z. Wahrsch. Verw. Gebiete}, 55(3):237--254, 1981.

\bibitem[Cal50a]{Cal50a}
A.P. Calder{\'o}n.
\newblock On a theorem of {M}arcinkiewicz and {Z}ygmund.
\newblock {\em Trans. Amer. Math. Soc.}, 68:55--61, 1950.

\bibitem[Cal50b]{Cal50b}
A.P. Calder{\'o}n.
\newblock On the behaviour of harmonic functions at the boundary.
\newblock {\em Trans. Amer. Math. Soc.}, 68:47--54, 1950.

\bibitem[Cao00]{Cao00}
J.~Cao.
\newblock Cheeger isoperimetric constants of {G}romov-hyperbolic spaces with
  quasi-poles.
\newblock {\em Commun. Contemp. Math.}, 2(4):511--533, 2000.

\bibitem[CY75]{CY75}
S.Y. Cheng and S.T. Yau.
\newblock Differential equations on {R}iemannian manifolds and their geometric
  applications.
\newblock {\em Comm. Pure Appl. Math.}, 28(3):333--354, 1975.

\bibitem[Doo57]{Doo57}
J.L. Doob.
\newblock Conditional {B}rownian motion and the boundary limits of harmonic
  functions.
\newblock {\em Bull. Soc. Math. France}, 85:431--458, 1957.

\bibitem[Dur84]{Dur84}
R.~Durrett.
\newblock {\em Brownian motion and martingales in analysis}.
\newblock Wadsworth Mathematics Series. Wadsworth International Group, Belmont,
  CA, 1984.

\bibitem[Fat06]{Fat06}
P.~Fatou.
\newblock S\'eries trigonom\'etriques et s\'eries de {T}aylor.
\newblock {\em Acta Math.}, 30(1):335--400, 1906.

\bibitem[GDLH90]{GdlH90}
{\'E}.~Ghys and P.~De~La~Harpe.
\newblock Espaces m\'etriques hyperboliques.
\newblock In {\em Sur les groupes hyperboliques d'apr\`es {M}ikhael {G}romov
  ({B}ern, 1988)}, volume~83 of {\em Progr. Math.}, pages 27--45. Birkh\"auser
  Boston, Boston, MA, 1990.

\bibitem[Gro81]{Gro81}
M.~Gromov.
\newblock Hyperbolic manifolds, groups and actions.
\newblock In {\em Riemann surfaces and related topics: {P}roceedings of the
  1978 {S}tony {B}rook {C}onference ({S}tate {U}niv. {N}ew {Y}ork, {S}tony
  {B}rook, {N}.{Y}., 1978)}, volume~97 of {\em Ann. of Math. Stud.}, pages
  183--213. Princeton Univ. Press, Princeton, N.J., 1981.

\bibitem[Gro87]{Gro87}
M.~Gromov.
\newblock Hyperbolic groups.
\newblock In {\em Essays in group theory}, volume~8 of {\em Math. Sci. Res.
  Inst. Publ.}, pages 75--263. Springer, New York, 1987.

\bibitem[Gun83]{Gun83}
R.F. Gundy.
\newblock The density of the area integral.
\newblock In {\em Conference on harmonic analysis in honor of {A}ntoni
  {Z}ygmund, {V}ol. {I}, {II} ({C}hicago, {I}ll., 1981)}, Wadsworth Math. Ser.,
  pages 138--149. Wadsworth, Belmont, CA, 1983.

\bibitem[Mou94]{Mou94}
F.~Mouton.
\newblock {\em Convergence Non-Tangentielle des Fonctions Harmoniques en
  Courbure N\'egatives}.
\newblock PhD thesis, Universit\'e Joseph Fourier, 1994.

\bibitem[Mou95]{Mou95}
F.~Mouton.
\newblock Comportement asymptotique des fonctions harmoniques en courbure
  n\'egative.
\newblock {\em Comment. Math. Helv.}, 70(3):475--505, 1995.

\bibitem[Mou00]{Mou00}
F.~Mouton.
\newblock Comportement asymptotique des fonctions harmoniques sur les arbres.
\newblock In {\em S\'eminaire de {P}robabilit\'es, {XXXIV}}, volume 1729 of
  {\em Lecture Notes in Math.}, pages 353--373. Springer, Berlin, 2000.

\bibitem[Mou07]{Mou07}
F.~Mouton.
\newblock Local {F}atou theorem and the density of energy on manifolds of
  negative curvature.
\newblock {\em Rev. Mat. Iberoam.}, 23(1):1--16, 2007.

\bibitem[Mou10]{Mou10}
F.~Mouton.
\newblock Non-tangential, radial and stochastic asymptotic properties of
  harmonic functions on trees.
\newblock {\em arXiv:1004.4416v1[math.MG]}, 2010.

\bibitem[MZ38]{MZ38}
J.~Marcinkiewicz and A.~Zygmund.
\newblock A theorem of {L}usin. {P}art {I}.
\newblock {\em Duke Math. J.}, 4(3):473--485, 1938.

\bibitem[Pet12]{Pet12}
C.~Petit.
\newblock Harmonic functions on hyperbolic graphs.
\newblock {\em Proc. Amer. Math. Soc.}, 140(1):235--248, 2012.

\bibitem[Pic10]{Pic10}
M.A. Picardello.
\newblock Local admissible convergence of harmonic functions on non-homogeneous
  trees.
\newblock {\em Colloq. Math.}, 118(2):419--444, 2010.

\bibitem[Pri16]{Pri16}
J.~Priwaloff.
\newblock Sur les fonctions conjugu\'ees.
\newblock {\em Bull. Soc. Math. France}, 44:100--103, 1916.

\bibitem[Spe43]{Spe43}
D.C. Spencer.
\newblock A function-theoretic identity.
\newblock {\em Amer. J. Math.}, 65:147--160, 1943.

\bibitem[Ste61]{Ste61}
E.M. Stein.
\newblock On the theory of harmonic functions of several variables. {II}.
  {B}ehavior near the boundary.
\newblock {\em Acta Math.}, 106:137--174, 1961.

\bibitem[Yau75]{Yau75}
S.T. Yau.
\newblock Harmonic functions on complete {R}iemannian manifolds.
\newblock {\em Comm. Pure Appl. Math.}, 28:201--228, 1975.

\end{thebibliography}
\bibliographystyle{alpha}

\end{document}